\documentclass[11pt]{article}

\usepackage{amsfonts}
\usepackage{amsmath}
\usepackage{amssymb}
\usepackage{amsthm}
\usepackage{graphicx}
\usepackage{indentfirst}
\usepackage{mathtools}

\usepackage{tocloft}
\usepackage{bm}

\usepackage{tikz}
\usetikzlibrary{arrows,shapes,chains}
\usepackage{caption}
\newtheorem{definition}{Definition}
\newtheorem{theorem}{Theorem}
\newtheorem{lemma}[theorem]{Lemma}
\newtheorem{corollary}[theorem]{Corollary}

\theoremstyle{plain}\newtheorem*{the 1}{Theorem 1}

\theoremstyle{definition}\newtheorem{case}{Case}
\theoremstyle{definition}\newtheorem*{case 1.1}{Case 1.1}
\theoremstyle{definition}\newtheorem*{case 1.2}{Case 1.2}
\theoremstyle{definition}\newtheorem*{case 2.1}{Case 2.1}
\theoremstyle{definition}\newtheorem*{case 2.2}{Case 2.2}
\theoremstyle{definition}\newtheorem*{case 3.1}{Case 3.1}
\theoremstyle{definition}\newtheorem*{case 3.2}{Case 3.2}
\theoremstyle{definition}\newtheorem*{case 3.3}{Case 3.3}
\theoremstyle{definition}\newtheorem*{case 3.3.1}{Case 3.3.1}
\theoremstyle{definition}\newtheorem*{case 3.3.2}{Case 3.3.2}

\title{A minimum semi-degree sufficient condition for one-to-many disjoint path covers in semicomplete digraphs}
\author{Ansong Ma$^{1}$,  Yuefang Sun$^{2,}$\footnote{Corresponding author. 
}, Xiaoyan Zhang$^{3}$
\\
$^{1}$ School of Mathematics and Statistics,
Ningbo University\\
Zhejiang 315211,  China, mas0710@163.com\\
$^{2}$ School of Mathematics and Statistics,
Ningbo University\\
Zhejiang 315211, China, sunyuefang@nbu.edu.cn\\
$^{3}$ School of Mathematical Science $\&$ Institute of Mathematics,\\Nanjing Normal University, Jiangsu 210023, China,\\ zhangxiaoyan@njnu.edu.cn}
\date{}
\begin{document}
	\maketitle
		\begin{abstract}
		 Let $D$ be a digraph. We define the minimum semi-degree of $D$ as $\delta^{0}(D) := \min \{\delta^{+}(D), \delta^{-}(D)\}$. Let $k$ be a positive integer, and let $S = \{s\}$ and $T = \{t_{1}, \dots ,t_{k}\}$ be any two disjoint subsets of $V(D)$. A set of $k$ internally disjoint paths joining source set $S$ and sink set $T$ that cover all vertices $D$ are called a one-to-many $k$-disjoint directed path cover ($k$-DDPC for short) of $D$. A digraph $D$ is semicomplete if for every pair $x,y$ of vertices of it, there is at least one arc between $x$ and $y$. 
		 
		 In this paper, we prove that every semicomplete digraph $D$ of sufficiently large order $n$ with $\delta^{0}(D) \geq \lceil (n+k-1)/2\rceil$ has a one-to-many $k$-DDPC joining any disjoint source set $S$ and sink set $T$, where $S = \{s\}, T = \{t_{1}, \dots, t_{k}\}$.
		\vspace{0.2cm} 
		\vspace{0.2cm}
	\end{abstract}

	\section{Introduction}
	For terminology and notation not defined here, we refer to \cite{Bang-Jensen08} and \cite{Bang-Jensen18}. 
	In this paper, a path always means a directed path. An {\it $x$-$y$ path} is a directed path which is from $x$ to $y$ for two vertices $x,y \in V(D)$. For two vertices $x$ and $y$ on a path $P$ satisfying $x$ precedes $y$, let $xPy$ denote the subpath of $P$ from $x$ to $y$. A digraph $D$ is {\it strongly connected}, or {\it strong} for short, if there exists an $x$-$y$ path and a $y$-$x$ path, for every pair $x$, $y$ of distinct vertices in $D$.
	
	Let $D = (V(D), A(D))$ be a digraph, and let $D_{1}$ and $D_{2}$ be two disjoint subdigraphs of $D$ satisfying $V(D_{1}) \cup V(D_{2}) = V(D)$. For $x \in V(D_{i})$, let $N^{+}_{D_{i}}(x):= \{y \in V(D_{i})\setminus x \colon xy \in A(D)\}$. $N^{-}_{D_{i}}(x)$ is defined similarly. Let $d^{+}_{D_{i}}(x) := \lvert N^{+}_{D_{i}}(x) \rvert$, $d^{-}_{D_{i}}(x) := \lvert N^{-}_{D_{i}}(x) \rvert$ and $N_{D_{i}}(x) := N^{+}_{D_{i}}(x) \cup N^{-}_{D_{i}}(x)$. 
	
	A {\it $k$-disjoint path cover} of an undirected graph $G$ is a set of $k$ internally disjoint paths connecting given disjoint source set and sink set such that all vertices of $G$ is covered by the path set. The $k$-disjoint path cover problem ($k$-DPC for short) has been studied by many researchers, see \cite{Jo13, Jo(1)13, Lim16, Liu10, Park04, Park06, Park08, Shih11, You15}. The problem of $k$-disjoint path cover can be classified into three types according to the number of elements in the source set and the sink set: one-to-one, one-to-many and many-to-many. The one-to-one type considers disjoint path covers joining a single pair of source $s$ and sink $t$ \cite{Cao18, Shih11, You15}, and the one-to-many type is about disjoint path covers which join a single source $s$ and a set of $k$ distinct sinks $t_{1}, t_{2}, \ldots , t_{k}$ \cite{Park04, Zhou}. The many-to-many type considers disjoint path covers between a set of $k$ sources $s_{1}, s_{2}, \ldots, s_{k}$ and another set of $k$ sinks $t_{1}, t_{2}, \ldots, t_{k}$ \cite{Jo13, Jo(1)13, Liu10, Park06, Park08}.
	
	Let $k$ be a positive integer, let $S = \{s\}$ and $T = \{t_{1}, \dots ,t_{k}\}$ be two disjoint subsets of $V(D)$. A set of disjoint path $\{P_{1}, \dots, P_{k}\}$ of $D$ is a {\it $k$-disjoint directed path cover} ({\em one-to-many $k$-DDPC} for short) of $D$, if $\bigcup^{k}_{i=1}V(P_{i}) = V(D)$ and $V(P_{i}) \cap V(P_{j}) = \{s\}$ for all $i \neq j$, where $P_{i}$ is a path from $s$ to $t_{i}$.
	
	A digraph $D$ is {\it semicomplete} if for every pair $x,y$ of vertices of it, there is at least one arc between $x$ and $y$. In this paper, we study the problem of one-to-many $k$-DDPC in semicomplete digraphs and prove the following main result. Note that our argument is inspired by that of \cite{Khn08}, 
	
	\begin{theorem}\label{The:main}
		
		Let $D$ be a semicomplete digraph of order $n\geq (9k)^{5}$, where $k$ $(\geq 2)$ is an integer. If $\delta^{0}(D)\geq \lceil (n+k-1)/2\rceil$, then $D$ has a one-to-many $k$-DDPC for any disjoint source set $S$ and sink set $T$, where $S = \{s\}, T = \{t_{1}, \dots, t_{k}\}$.
		
	\end{theorem}
	
	\section{Preliminaries}
	
	We now introduce two results concerning the existence of a Hamiltonian path (cycle) in semicomplete digraphs.
	
	\begin{theorem} \cite{Redei1934} \label{The:Redei}
		Every semicomplete digraph contains a Hamiltonian path.
		
	\end{theorem}
	
	\begin{theorem} \cite{Camion1959} \label{The:Camion}
	Every strong semicomplete digraph on $n\geq 3$ vertices has a Hamiltonian cycle.   	
	\end{theorem}

	The following is the definition of {\it $H$-subdivision}.
		\begin{definition} \cite{Ferrara06} \label{Def:division}
		Let $H$ be a (multi)digraph and $D$ be a digraph. Let $\mathcal{P}(D)$ denote the set of paths in $D$. An {\em $H$-subdivision} in $D$ is a pair of mappings $f \colon V(H) \rightarrow V(D)$ and $g \colon A(H) \rightarrow \mathcal{P}(D)$ such that:
		
		(a) $f(u) \neq f(v)$ for all disjoint $u, v\in V(H)$ and
		
		(b) for every $uv \in A(H)$, $g(uv)$ is an $f(u)$-$f(v)$ path in $D$, and disjoint arcs map into internally disjoint paths in $D$.
		
	\end{definition}
	
	A digraph $D$ is {\it $H$-linked} if every injective mapping $f \colon V(H) \rightarrow V(D)$ can be extended to an $H$-subdivision in $D$. A digraph $D$ is {\em one-to-many $k$-linked} if it has a set of $k$  paths $\{P_{1}, \ldots, P_{k}\}$ for any disjoint vertex subsets $S=\{s\}$ and $T=\{t_{1},\dots,t_{k}\}$ such that every path $P_{i}$ is from $s$ to $t_{i}$ and $V(P_{i})\cap V(P_{j}) = \{s\}$ for $i\neq j$. 
	
	Ferrara, Jacobson and Pfender \cite{Ferrara13} gave a sufficient condition for a digraph $D$ to be $H$-linked. By this result, Zhou~\cite{Zhou} found a specific digraph $H$ and got the following corollary.
	
	
	\begin{corollary} \cite{Zhou} \label{Cor-Zhou}
		Let $k$ $(\geq 2)$ be an integer, and let $D$ be a digraph with order $n \geq 80k$. If  $\delta^{0}(D)\geq \lceil (n + k - 1)/2\rceil$, then $D$ is one-to-many $k$-linked.
	\end{corollary}
	
	\section{Proof of Theorem 1}
	In the rest of this paper, let $D$ be a semicomplete digraph which satisfies the assumption of Theorem~\ref{The:main}. Let $S = \{s\}$ and $T = \{t_{1},\ldots ,t_{k}\}$ be any two disjoint subsets of $D$. An {\em $S$-$T$ path} in $D$ is a set of $k$-disjoint paths $P_{1},\dots ,P_{k}$, where every path $P_{i}$ is from $s$ to $t_{i}$ and $V(P_{i})\cap V(P_{j}) = \{s\}$ for all $i \neq j$. By Corollary~\ref{Cor-Zhou}, $D$ contains at least one $S$-$T$ path. Let $L$ be a maximum $S$-$T$ path, that is, it covers the most vertices of $D$ among all $S$-$T$ paths. Suppose that $L$ is not a one-to-many $k$-DDPC (this implies that $\bigcup^{k}_{i=1}V(P_{i}) \subsetneqq V(D)$). We aim to obtain a larger $S$-$T$ path from $L$ and hence this produces a contradiction. 
	
	Let $H$ be the subdigraph of $D$ induced by $V(D) \setminus V(L)$. Since $D$ is semicomplete, both $L$ and $H$ are semicomplete. Let $F := \{x \in V(L) \colon \exists  y \in V(H) \ \mathrm{such} \ \mathrm{that} \ yx \in A(D)\}$, $R := \{x \in V(L) \colon \exists  y \in V(H) \ \mathrm{such} \ \mathrm{that} \ xy \in A(D)\}$. Let $F_{m} := F \setminus R$ and $R_{m} := R \setminus F$. Since $D$ is semicomplete, for every vertex $h \in V(H)$ and every vertex $x \in R_{m}$, there is an arc from $x$ to $h$. Similarly, for every vertex $y \in F_{m}$, there is an arc from $h$ to $y$. For a vertex $v \in V(P_{i}) \setminus \{s\}$, its predecessor is denoted by $v^{-} \in V(P_{i}) \setminus \{t_{i}\}$. Similarly, for a vertex $v \in V(P_{i}) \setminus \{t_{i}\}$, its successor is denoted by $v^{+} \in V(P_{i}) \setminus \{s\}$. Let $F^{-} := \{x^{-} \colon x \in F \setminus S\}$ and $ R^{+} := \{ x^{+} \colon x \in R \setminus T\}$. Similarly, let $F^{-}_{m} := \{x^{-} \colon x \in F_{m} \setminus S\}$ and $R^{+}_{m} := \{x^{+} \colon x \in R_{m} \setminus T\}$. To prove Theorem 1, we need more preliminary results, including several lemmas and corollaries.
	
	\begin{lemma}\label{Lem:connect}
		The following assertions hold:
\begin{description}
\item[(a) ] $d^{-}_{H}(x) + d^{+}_{H}(y) \geq \lvert H \rvert - 1$ for every pair of vertices $x,y \in V(H)$. 
\item[(b) ] $H$ is strong with $\lvert H \rvert \leq \lfloor \frac{n-k+1}{2}\rfloor - 1$.
\end{description}
	\end{lemma}
	\begin{proof}
		For any vertex $x \in V(H)$, we have $d^{-}_{H}(x) + d^{+}_{H}(x) \geq \lvert H \rvert - 1$ since $H$ is semicomplete. The case that $|H|=1$ is trivial, so now we assume that $|H|\geq 2$. 
		
		We first consider the case that $\lvert H \rvert = 2$, say $V(H) = \{x, y\}$. There is at least one arc, say $xy\in A(H)$, between $x$ and $y$ since $H$ is semicomplete. We claim that $d^{-}_{L}(x) + d^{+}_{L}(y) \leq \lvert L \rvert + k$. Suppose that $d^{-}_{L}(x) + d^{+}_{L}(y) \geq \lvert L \rvert + k + 1$. Let $X = N^{-}_{L}(x)$ and $Y = \{ v \colon v^{+} \in N^{+}_{L}(y)\}$. If $s \in N^{+}_{L}(y)$, then $s^{-}$ does not exist. For each path $P_{i}$, if $s^{+} \in N^{+}_{L}(y)$, then there is only one $s$ satisfying $s \in Y$. Therefore, we have $\lvert Y \rvert \geq d^{+}_{L}(y) - k$ and
		$$\lvert X \cap Y \rvert = \lvert X \rvert + \lvert Y \rvert - \lvert X \cup Y \rvert \geq d^{-}_{L}(x) + d^{+}_{L}(y) - k - \lvert L \rvert \geq \lvert L \rvert + k + 1 - k - \lvert L \rvert = 1.$$
		This implies that there are two vertices $v \in N^{-}_{L}(x)$ and $v^{+} \in N^{+}_{L}(y)$. By replacing the arc $vv^{+}$ with the path $vxyv^{+}$, we obtain a larger $S$-$T$ path, a contradiction. Therefore, $d^{-}_{L}(x) + d^{+}_{L}(y) \leq \lvert L \rvert + k$ and 
		$$d^{-}_{H}(x) + d^{+}_{H}(y) \geq 2\lceil \frac{n+k-1}{2} \rceil - (\lvert L \rvert + k) \geq \lvert H \rvert - 1 = 1,$$ which satisfies the degree condition in Lemma \ref{Lem:connect}. Hence, $yx\in A(H)$ and so $H$ is strong in this case. 
		
		When $\lvert H \rvert \geq 3$, from the argument above, any two vertices $x_{1}, y_{1} \in V(H)$ with $x_{1}y_{1} \in A(D)$ satisfy the degree condition in Lemma~\ref{Lem:connect}. If $y_{1}x_{1} \notin A(H)$, then $y_{1} \notin N^{-}_{H}(x_{1})$ and $x_{1} \notin N^{+}_{H}(y_{1})$. Consequently, $\lvert N^{-}_{H}(x_{1}) \cup N^{+}_{H}(y_{1}) \rvert \leq \lvert H \rvert - 2$, and so $\lvert N^{-}_{H}(x_{1}) \cap N^{+}_{H}(y_{1}) \rvert = \lvert N^{-}_{H}(x_{1}) \rvert + \lvert N^{+}_{H}(y_{1}) \rvert - \lvert N^{-}_{H}(x_{1}) \cup N^{+}_{H}(y_{1}) \rvert \geq d^{-}_{H}(x_{1}) + d^{+}_{H}(y_{1}) - (\lvert H \rvert - 2) \geq 1$. This implies that $H$ contains a $y_{1}$-$x_{1}$ path of length 2. Therefore, when $x_{1}y_{1}\in A(H)$, either $y_{1}x_{1} \in A(H)$ or there is a $y_{1}$-$x_{1}$ path of length 2 in $H$. For any vertex $x \in V(H)$, we have $\lvert N_{H}(x) \rvert = \lvert N^{-}_{H}(x) \cup N^{+}_{H}(x) \rvert = \lvert H \rvert - 1$ since $H$ is semicomplete. According to the argument above, there are two paths of both directions between $x$ and every vertex in $N_{H}(x)$. For any two vertices $x, z \in V(H)$, note that $z \in N_{H}(x)$ since $H$ is  semicomplete, and this implies that there are two paths of both directions between $x$ and $z$. Hence, $H$ is also strong in this case.
		
		For every pair of vertices $x,y \in V(H)$, there is an $x$-$y$ path, say $P$, since $H$ is strong. With a similar argument to that of the case $|H|=2$, we can get $d^{-}_{L}(x) + d^{+}_{L}(y) \leq \lvert L \rvert + k$, and for every pair of vertices $x,y \in V(H)$,
		\begin{equation}
			d^{-}_{H}(x) + d^{+}_{H}(y) \geq \lvert H \rvert - 1. \label{equ:degree}
		\end{equation}
		In particular, when $V(H) = \{h\}$, $d^{-}_{H}(h) + d^{+}_{H}(h) = \lvert H \rvert - 1 = 0$ (satisfying the above degree condition).
		
We now prove the upper bound for $\lvert H \rvert$. Note that $R^{+} \cap F = \emptyset$. If $v\in R^{+} \cap F$, then $v^{-} \in R$. There are two vertices $u_{1}, u_{2} \in V(H)$ such that $v^{-}u_{1}, u_{2}v \in A(D)$. Since $H$ is strong, $H$ contains a $u_{1}$-$u_{2}$ path $P$. By replacing the arc $v^{-}v$ with the path $v^{-}Pv$, we obtain a larger $S$-$T$ path, a contradiction. This implies that there is a vertex $v \in V(L)$ such that $N^{+}_{D}(v) \subseteq L$ or $N^{-}_{D}(v) \subseteq L$. Thus $\lvert L \rvert - 1 \geq \delta^{0}(D) \geq \lceil(n+k-1)/2\rceil$,  and then $\lvert H \rvert \leq \lfloor(n-k+1)/2\rfloor -1 $. \qedhere
		
	\end{proof}
	
	According to the proof above, we have $R^{+} \cap F = \emptyset$. Similarly, $R \cap F^{-} =\emptyset$, so the following corollary holds.
	
	\begin{corollary}\label{Cor:RF}
		For any $u \in R$, $u^{+} \notin F$. Similarly, for any $v \in F$, $v^{-} \notin R$.
	\end{corollary}
	
	\begin{lemma}\label{Lem:FpR}
	The following assertions hold:
\begin{description}
\item[(a) ]  $\lvert F \cup R \rvert = \lvert L \rvert$ and $\lvert F \cap R \rvert \leq k$.
\item[(b) ] For any path $P_{i}$, if $x_{1}$ precedes $x_{2}$ in $P_{i}$, then there are no distinct vertices $y_{1}, y_{2} \in V(H)$ such that $x_{1}y_{1}, y_{2}x_{2} \in A(D)$.
\end{description}
	\end{lemma}
	
	\begin{proof}
		For any vertex $x\in V(L)$ and any vertex $y \in V(H)$, there is at least one arc between $x$ and $y$ since $H$ is semicomplete. Observe that $x \in F$, or $x \in R$ (or both) for any vertex $x \in L$, so $\lvert F \cup R \rvert = \lvert L \rvert$.
		
		For each path $P_{i}$, we claim that $\lvert (R \cap F) \cap V(P_{i}) \rvert \leq 1$. Suppose that $u, v \in R \cap F$ are two distinct vertices in $P_{i}$ and $u$ precedes $v$. Let $Q_{1}$ denote the set of all vertices between $u$ and $v$. For any path $P_{i}$, if $x \notin F$, then $x \in R$ since $\lvert F \cup R \rvert = \lvert L \rvert$. By Corollary~\ref{Cor:RF}, $u^{+} \notin F$, for any $u \in R$. Therefore, $u^{+} \in R$ and $V(u^{+}P_{i}t_{i}) \subseteq R$. As $v^{-}\in V(u^{+}P_{i}t_{i})$, we get that $v^{-}\in R$. By the assumption, $v\in F$, so there are two vertices $h_{1}, h_{2} \in V(H)$ such that $v^{-}h_{1}, h_{2}v \in A(D)$. By Lemma~\ref{Lem:connect}, $H$ is strong and hence contains an $h_{1}$-$h_{2}$ path $P$. By replacing $v^{-}v$ with $v^{-}Pv$, we obtain a larger $S$-$T$ path, a contradiction. Particularly, when $Q_{1} = \emptyset$, $u \in R$ and $v = u^{+} \in F$ by the assumption. Similar to the argument above, we still can find a larger $S$-$T$ path, this also produces a contradiction. Consequently, $|(R \cap F) \cap V(P_{i})| \leq 1$ for each path $P_{i}$, which implies that $\lvert R \cap F \rvert \leq k$.
		
		Now we prove the assertion \textbf(b). Assume that such $y_{1}, y_{2}$ exist. Therefore, $x_{1} \in R$ and $x_{2} \in F$ in $P_{i}$. Similar to the proof above, let $Q_{2}$ denote the set of all vertices between $x_{1}$ and $x_{2}$. For any path $P_{i}$, if $x \notin F$, then $x \in R$ since $\lvert F \cup R \rvert = \lvert L \rvert$. By Corollary~\ref{Cor:RF}, $u^{+} \notin F$, for any $u \in R$. Therefore, $x^{+}_{1} \in R$, $V(x^{+}_{1}P_{i}t_{i}) \subseteq R$ and $x^{-}_{2}\in R$. Together with the fact $x_{2}\in F$, we can find a larger $S$-$T$ path, a contradiction. Particularly, if $Q_{2}=\emptyset$, then $x_{1} \in R$ and $x_{2} = x^{+}_{1} \in F$ by the assumption. Similar to the argument above, we still find a larger $S$-$T$ path, and hence this also produces a contradiction. Consequently, for any path $P_{i}$, if $x_{1}$ precedes $x_{2}$ in $P_{i}$, then there are no distinct vertices $y_{1}, y_{2} \in H$ such that $x_{1}y_{1}, y_{2}x_{2} \in A(D)$. This implies $x_{2}$ always precedes $x_{1}$ for any
		$x_{1} \in R, x_{2} \in F$ in each path $P_{i}$. \qedhere
	\end{proof}
	
	The proof of the following lemma is similar to that of Lemma 13 in \cite{{Khn08}}, but we still need to give the proof below.
	
	\begin{lemma}\label{Lem:bound}
		$\lvert F \rvert$, $\lvert R \rvert \geq (n + k + 1)  /2 -\lvert H \rvert$. Furthermore, $\lvert F_{m} \rvert$, $\lvert R_{m} \rvert \geq (n - k + 1)/2 -\lvert H \rvert$ and $\vert F_{m} \cup R_{m}\rvert \geq \lvert L \rvert - k$.
	\end{lemma}
	
	\begin{proof}
		For every vertex $x \in V(H)$, $d^{-}_{L}(x) \geq \delta^{0}(D) - (\lvert H \rvert - 1) \geq (n + k + 1)/ 2 - \lvert H \rvert$, and so $\lvert R \rvert \geq (n + k + 1) / 2 - \lvert H \rvert$. The inequality $\lvert F \rvert \geq (n + k + 1)  /2 -\lvert H \rvert$ can be proved similarly. By Lemma~\ref{Lem:FpR}, $\lvert R \cap F\rvert \leq k$, so $\lvert R_{m} \rvert = \lvert R \rvert - \lvert R \cap F \rvert \geq (n + k + 1) / 2 -\lvert H \rvert - k = (n - k + 1) / 2 -\lvert H \rvert$. Similarly, we can prove the inequality $\lvert F_{m} \rvert \geq (n - k + 1)/2 -\lvert H \rvert$. 
		
		By Lemma~\ref{Lem:FpR}, $\lvert F \cup R \rvert = \lvert L \rvert$ and $\lvert F \cap R \rvert \leq k$. Together with the definition of $F_{m}$ and $R_{m}$, $\lvert F \cup R \rvert = \lvert F_{m} \rvert + \lvert R_{m} \rvert + \lvert F \cap R \rvert$, so $\lvert R_{m} \cup F_{m} \rvert = \lvert F_{m} \rvert + \lvert R_{m} \rvert = \lvert F \cup R \rvert - \lvert F \cap R \rvert \geq \lvert L \rvert - k$. \qedhere
		\end{proof}	
	
	We rewrite Theorem~\ref{The:main} and give the proof now:
	
	\begin{the 1}
		
		Let $D$ be a semicomplete digraph of order $n\geq (9k)^{5}$, where $k$ $(\geq 2)$ is an integer. If $\delta^{0}(D)\geq \lceil (n+k-1)/2\rceil$, then $D$ has a one-to-many $k$-DDPC for any disjoint source set $S$ and sink set $T$, where $S = \{s\}, T = \{t_{1}, \dots, t_{k}\}$.
		
	\end{the 1}
	
	\begin{proof}
		
		Recall that at the beginning of Section 3, we supposed that $L$ is not a one-to-many $k$-DDPC. It suffices to find a larger $S$-$T$ path more than $L$, and then this produces a contradiction and thus we prove Theorem~\ref{The:main}. Our argument is divided into the following three cases.
		
		\begin{case}
			$\lvert H \rvert = 1$.
		\end{case}
		Let $V(H)$ = $\{h\}$. Since $\lvert H \rvert = 1$, we have $d^{-}_{L}(h) = \lvert R \rvert$, $d^{+}_{L}(h) = \lvert F \rvert$ and $d^{+}_{L}(h) + d^{-}_{L}(h) = d^{+}_{D}(h) + d^{-}_{D}(h)$. According to whether $n + k$ is even, we distinguish the following two subcases.
		
		\begin{case 1.1}
			$n + k$ is even.
		\end{case 1.1}
		When $n + k$ is even, we have $\lceil \frac{n + k - 1}{2}\rceil = \frac{n + k}{2}$ and $d^{+}_{L}(h) + d^{-}_{L}(h) \geq 2\delta^{0}(D) \geq n +k$. Consequently, $\lvert F \rvert + \lvert R \rvert \geq n +k$. By Lemma~\ref{Lem:FpR}, $\lvert R \cup F \rvert = \lvert L \rvert = n - 1$. Therefore, $\lvert R \cap F \rvert = \lvert R \rvert + \lvert F \rvert - \lvert R \cup F \rvert \geq (n + k) - (n -1) = k + 1$, which implies that there exists a path $P_{i}$ such that $u, v \in R \cap F$. Without loss of generality, assume that $u$ precedes $v$. Arguing similarly as in the proof of Lemma~\ref{Lem:FpR}, we can find a larger $S$-$T$ path, say $L^{*}$, which contains $h$. Consequently, $L^{*}$ is a one-to-many $k$-DDPC in $D$, a contradiction.
		
		\begin{case 1.2}
			$n+k$ is odd.
		\end{case 1.2}
		
		When $n+k$ is odd, we have $\lceil(n+k-1) / 2\rceil = (n+k-1) / 2$. Therefore, $d^{+}_{L}(h) = d^{+}_{D}(h) \geq (n+k-1) / 2$, $d^{-}_{L}(h) = d^{-}_{D}(h) \geq (n+k-1) / 2$, and $d^{+}_{L}(h) + d^{-}_{L}(h) \geq 2\delta^{0}(D) \geq n+k-1$. Consequently, $\lvert R \rvert + \lvert F \rvert \geq n + k -1$. Since $\lvert L \rvert = n-1$, $\lvert R \cap F \rvert = \lvert R \rvert + \lvert F \rvert - \lvert R \cup F \rvert \geq (n+k-1) - (n-1) = k$. By Lemma~\ref{Lem:FpR}, $\lvert R \cap F\rvert \leq k$, so $\lvert R \cap F\rvert = k$. According to the argument in Lemma~\ref{Lem:connect}, $d^{+}_{L}(h) + d^{-}_{L}(h) \leq \lvert L \rvert + k = n+k-1$. Therefore, $d^{+}_{L}(h) + d^{-}_{L}(h) = n+k-1$, and then $d^{+}_{L}(h) = d^{-}_{L}(h) = (n+k-1) / 2$. Thus we get $\lvert F \rvert = d^{+}_{L}(h) = (n+k-1) / 2$, $\lvert R \rvert = d^{-}_{L}(h) = (n+k-1) / 2$ and $\lvert R_{m} \rvert = \lvert F_{m} \rvert = (n+k-1) / 2 - k = (n-k-1) / 2$ by Lemma~\ref{Lem:bound}.
		
		Note that each $y \in F_{m}$ satisfies $d^{+}_{R_{m}}(y) \geq \delta^{0}(D) - (\lvert F \rvert - 1) = 1$, so there is at least one vertex $x_{1} \in R_{m}$ such that $yx_{1} \in A(D)$. By Lemma~\ref{Lem:FpR}, we have $s \in F_{m}$, so there exists a vertex $x_{1} \in R_{m}$ in $P_{i}$ such that $sx_{1} \in A(D)$. As $\lvert (R \cap F) \cap V(P_{i}) \rvert \leq 1$ and $\lvert R \cap F \rvert = k$, there is a vertex $s^{+}_{j} \in F$ in $P_{j}$. Together with the fact that $x^{-}_{1} \in R$, there exists an $x^{-}_{1}$-$s^{+}_{j}$ path whose inner vertex is $h$. Let $P^{*}_{i} := sx_{1}P_{i}t_{i}$ and $P^{*}_{j} := sP_{i}x^{-}_{1}hs^{+}_{j}P_{j}t_{j}$. Therefore, we obtain a
		larger $S$-$T$ path, say $L^{*}$, which contains $h$. Consequently, $L^{*}$ is a one-to-many $k$-DDPC in $D$, this produces a contradiction (see Figure~\ref{fig:odd}).
		
		\begin{figure}[htb]
			\centering
			\begin{tikzpicture}
				\filldraw[black]    (0, 0)  circle (3pt)  node [anchor=east] {$s$};
				\filldraw[black]    (4, 1)  circle (3pt)  node [anchor=north] {$x^{-}_{1}$};
				\filldraw[black]    (5, 1.1)  circle (3pt)  node [anchor=north] {$x_{1}$};
				\filldraw[black]    (8, 1.2)  circle (3pt)  node [anchor=west] {$t_{i}$};
				\filldraw[black]    (1, -0.9)  circle (3pt)  node [anchor=north] {$s^{+}_{j}$};
				\filldraw[black]    (8, -1.2)  circle (3pt)  node [anchor=west] {$t_{j}$};
				\filldraw[black]    (2.5, 0)  circle (3pt)  node [anchor=east] {$h$};
				
				\draw[line width=2pt] [-latex]      (0,0) .. controls (1, 1) and (2, 1) .. (4, 1) ;
				\draw[line width=2pt] [-latex]     (0,0) .. controls (2, 2) and (3, 2) .. (5, 1.1) ;
				\draw[line width=2pt] [-latex]      (5, 1.1) ..controls (6, 1.2) and (7, 1.2) .. (8, 1.2) ;
				\draw[] [-latex]     (0,0) .. controls (0.4,-0.6) and (0.5, -0.6) .. (1,-0.9) ;
				\draw[line width=2pt] [-latex]    (1,-0.9) .. controls (2.5,-1.2) and (6.5, -1.2) .. (8,-1.2) ;
				
				\draw[line width=2pt] [-latex]     (4, 1) -- (2.5, 0) ;
				\draw[line width=2pt] [-latex]     (2.5, 0) -- (1, -0.9) ;
				\draw[] [-latex]      (4, 1) -- (5, 1. 1) ;
				
			\end{tikzpicture}
			\caption{}
			\label{fig:odd}
		\end{figure}
		
		\begin{case}
			$2 \leq \lvert H \rvert \leq n / 2 - n / (50k)$.
		\end{case}
		
		By Lemma~\ref{Lem:bound}, $\lvert R_{m} \rvert \geq (n - k + 1) / 2 -\lvert H \rvert \geq n / (53k)$, so there exists a path $P_{j}$ which contains at least $n / (53k^{2})$ vertices from $R_{m}$. Without loss of generality, assume that $j = 1$. Let $A$ be the set of vertices belonging to $R_{m}$ in $P_{1}$. We use $A_{1}$ ($A_{2}$, respectively) to denote the subpath of $A$ which contains the first (last, respectively) $n / (110k^{2})$ inner vertices of $A$.
		
		The first vertex of $A$ is denoted by $t$. For any vertex $a\in V(t^{+}P_{1}t_{1})$, Lemma~\ref{Lem:FpR} implies that $a \in R$ and $N^{-}_{D}(a) \subseteq V(L)$. By Lemma~\ref{Lem:bound}, $\lvert F \rvert \geq (n + k + 1) / 2 -\lvert H \rvert$ and so 
		\begin{equation}
			d^{-}_{L}(a) \geq \delta^{0}(D) \geq (n+k-1) / 2 \geq n + k - \lvert H \rvert - \lvert F \rvert = \lvert L \rvert - \lvert F \rvert + k. \label{equ:ine} 
		\end{equation}
		
		\begin{case 2.1}
			There are two vertices $a_{1} \in A_{1}$ and $a_{2} \in A_{2}$ such that $a_{1}a_{2} \in A(D)$.
			
		\end{case 2.1}
		
		By inequality~\eqref{equ:ine}, we have $d^{-}_{L}(a^{+}_{1}) \geq \lvert L \rvert - \lvert F \rvert + k$ and 
		\begin{equation}
			\begin{split}
				\lvert N^{-}_{L}(a^{+}_{1}) \cap F^{-} \rvert &= \lvert N^{-}_{L}(a^{+}_{1}) \rvert + \lvert F^{-} \rvert - \lvert N^{-}_{L}(a^{+}_{1}) \cup F^{-}\rvert \\
				&\geq d^{-}_{L}(a^{+}_{1}) + (\lvert F \rvert - k) - (\lvert L \rvert - 1) \\
				&\geq \lvert L \rvert - \lvert F \rvert + k + (\lvert F \rvert - k) - (\lvert L \rvert - 1) = 1 \nonumber
			\end{split}
		\end{equation}
		(By the definition of $F^{-}$, we have $F^{-} \cap T = \emptyset$. Corollary~\ref{Cor:RF} implies that $F^{-} \cap R = \emptyset$ and $a^{+}_{1} \notin F^{-}$, thus $\lvert N^{-}_{L}(a^{+}_{1}) \cup F^{-}\rvert \leq \lvert L \rvert - 1$. If $s \in F$, then $s^{-}$ does not exist. For each $P_{i}$, if $s^{+}_{i} \in F$, then there is only one $s$ satisfying $s \in F^{-}$. Therefore, $\lvert F^{-} \rvert \geq \lvert F \rvert - k$). This implies that there are two vertices $w \in N^{-}_{L}(a^{+}_{1}) \cap F^{-}$ and $w^{+} \in F$. Lemma~\ref{Lem:FpR} implies that $F \cap V(t^{+}P_{1}t_{1}) = \emptyset$, so $w^{+} \in V(s^{+}P_{1}t)$ or $w^{+} \in V(P_{i}\setminus \{s\})$ $(i \neq 1)$. Therefore, $w \in V(sP_{i}t^{-})$ or $w \in V(p_{i} \setminus \{t_{i}\})$ $(i \neq 1)$. By Lemma~\ref{Lem:connect}, $H$ is strong. When $\lvert H \rvert \geq 3$, by Theorem~\ref{The:Camion}, we get that $H$ contains a Hamiltonian cycle, say $C$. For a vertex $x$ on $C$, its predecessor on $C$ is denoted by $x^{-}$ and its successor on $C$ is denoted by $x^{+}$. When $\lvert H \rvert \geq 3$, there exists a vertex $u \in V(H)$ such that $uw^{+} \in A(D)$ since $w^{+} \in F$. As $a^{-}_{2} \in R_{m}$, there is an arc from $a^{-}_{2}$ to $u^{+}$. Note that there is a Hamiltonian path from $u^{+}$ to $u$ in $H$. When $\lvert H \rvert = 2$, let $V(H)$ = $\{u, v\}$. There exists a vertex, say $v \in V(H)$, such that $vw^{+} \in A(D)$ since $w^{+} \in F$. As $a^{-}_{2} \in R_{m}$, there is an arc from $a^{-}_{2}$ to $u$. Note that $uv, vu \in A(H)$ since $H$ is strong.
		
		If $w \in V(sP_{1}t^{-})$, then according to the argument above, there exists a path $a^{-}_{2}Q_{1}w^{+}$, where $Q_{1}$ contains all vertices in $H$. Let $P^{*}_{1} := sP_{1}wa^{+}_{1}P_{1}a^{-}_{2}Q_{1}w^{+}P_{1}a_{1}a_{2}P_{1}t_{1}$. Now we obtain a larger $S$-$T$ path $L^{*}$ which is a one-to-many $k$-DDPC in $D$. This produces a contradiction (see Figure~\ref{fig:spt}). 
		\begin{figure}[htb]
			\centering
			\begin{tikzpicture}
				\filldraw[black]    (0, 0)  circle (3pt)  node [anchor=north] {$s$};
				\filldraw[black]    (1.5, 0)  circle (3pt)  node [anchor=south] {$w$};
				\filldraw[black]    (2.5, 0)  circle (3pt)  node [anchor=north] {$w^{+}$};
				\filldraw[black]    (4, 0)  circle (3pt)  node [anchor=south] {$a_{1}$};
				\filldraw[black]    (5, 0)  circle (3pt)  node [anchor=south] {$a^{+}_{1}$};
				\filldraw[black]    (6.5, 0)  circle (3pt)  node [anchor=north] {$a^{-}_{2}$};
				\filldraw[black]    (7.5, 0)  circle (3pt)  node [anchor=south] {$a_{2}$};
				\filldraw[black]    (9, 0)  circle (3pt)  node [anchor=west] {$t_{1}$};
				
				\draw[line width=2pt] [-latex]      (1.5, 0) .. controls (2.5, -1.5) and (4, -1.5) .. (5, 0) ;
				\draw[line width=2pt] [-latex]     (4, 0) .. controls (5, -1.5) and (6.5, -1.5) .. (7.5, 0) ;
				\draw[line width=2pt] [-latex]     (6.5, 0) .. controls (5, 1.5) and (4, 1.5) .. (2.5, 0) node[midway, sloped, above] {$Q_{1}$};
				
				\draw[line width=2pt] [-latex]     (0, 0) -- (1.5, 0) ;
				\draw[] [-latex]      (1.5, 0) -- (2.5, 0) ;
				\draw[line width=2pt] [-latex]     (2.5, 0) -- (4, 0) ;
				\draw[] [-latex]      (4, 0) -- (5, 0) ;
				\draw[line width=2pt] [-latex]      (5, 0) -- (6.5, 0) ;
				\draw[] [-latex]      (6.5, 0) -- (7.5, 0) ;
				\draw[line width=2pt] [-latex]      (7.5, 0) -- (9, 0) ;
				
			\end{tikzpicture}
			
			\caption{}
			\label{fig:spt}
		\end{figure}
		
		If $w \in V(P_{i}\setminus \{t_{i}\})$ $(i \neq 1)$, then according to the argument above, we can also find a path $a^{-}_{2}Q_{2}w^{+}$, where $Q_{2}$ contains all vertices in $H$. Let $P^{*}_{1} := sP_{1}a_{1}a_{2}P_{1}t_{1}$ and $P^{*}_{i} := sP_{i}wa^{+}_{1}P_{1}a^{-}_{2}Q_{2}w^{+}P_{i}t_{i}$. Therefore, we obtain a larger $S$-$T$ path $L^{*}$ which is a one-to-many $k$-DDPC in $D$. This produces a contradiction (see Figure~\ref{fig:pt}).
		
		\begin{figure}[htb]
			\centering
			\begin{tikzpicture}
				
				\filldraw[black]    (0, 0)  circle (3pt)  node [anchor=east] {$s$};
				\filldraw[black]    (2, 0.9)  circle (3pt)  node [anchor=north] {$a_{1}$};
				\filldraw[black]    (3, 1)  circle (3pt)  node [anchor=south] {$a^{+}_{1}$};
				\filldraw[black]    (6, 1.1)  circle (3pt)  node [anchor=south] {$a^{-}_{2}$};
				\filldraw[black]    (7, 1.1)  circle (3pt)  node [anchor=north] {$a_{2}$};
				\filldraw[black]    (8, 1.1)  circle (3pt)  node [anchor=west] {$t_{1}$};
				\filldraw[black]    (4, -1.1)  circle (3pt)  node [anchor=north] {$w$};
				\filldraw[black]    (5, -1.1)  circle (3pt)  node [anchor=north] {$w^{+}$};
				\filldraw[black]    (8, -1.1)  circle (3pt)  node [anchor=west] {$t_{i}$};
				
				\draw[line width=2pt] [-latex]      (0, 0) .. controls (0.5, 0.9) and (1.5, 0.9) .. (2, 0.9) ;
				\draw[line width=2pt] [-latex]     (2, 0.9) .. controls (3.5, 2.5) and (5.5, 2.5) .. (7, 1.1) ;
				\draw[line width=2pt] [-latex]     (0, 0) .. controls (1, -1.1) and (3, -1.1) .. (4, -1.1) ;
				\draw[line width=2pt] [-latex]    (4, -1.1) .. controls (3, -0.5) and (3, 0.6) .. (3, 1) ;
				\draw[line width=2pt] [-latex]    (6, 1.1) .. controls (6, 0.5) and (6, -0.3) .. (5, -1.1) node[midway, right] {$Q_{2}$};
				
				\draw[] [-latex]     (2, 0.9) -- (3, 1) ;
				\draw[line width=2pt] [-latex]     (3, 1) -- (6, 1.1) ;
				\draw[] [-latex]      (6, 1.1) -- (7, 1.1) ;
				\draw[line width=2pt] [-latex]      (7, 1.1) -- (8, 1.1) ;
				\draw[] [-latex]      (4, -1.1) -- (5, -1.1) ;
				\draw[line width=2pt] [-latex]      (5, -1.1) -- (8, -1.1) ;
				
			\end{tikzpicture}
			
			\caption{}
			\label{fig:pt}
		\end{figure}
		
		\begin{case 2.2}
			Case 2.1 does not hold.
		\end{case 2.2}
		
		By the definition of $F^{-}_{m}$, $F^{-}_{m} \cap V(tP_{1}t_{1}) = \emptyset$. For any vertex $a \in A_{2}$, Lemma~\ref{Lem:FpR} implies that $N^{-}_{D}(a) \subseteq V(L)$. By the assumption, $N^{-}_{D}(a) \subseteq V(L)\setminus V(A_{1})$. By Lemma~\ref{Lem:bound}, $\lvert F_{m} \rvert \geq (n - k + 1) / 2 -\lvert H \rvert$. 
		Similar to the argument in \eqref{equ:ine}, we get that $d^{-}_{L-A_{1}}(a) \geq \delta^{0}(D) \geq (n + k -1) / 2 \geq n - \lvert H \rvert - \lvert F_{m} \rvert \geq \lvert L - A_{1} \rvert + \lvert A_{1} \rvert - \lvert F^{-}_{m} \rvert - k$ (recall that $\lvert F^{-} \rvert \geq \lvert F \rvert - k$, so $\lvert F \rvert \leq \lvert F^{-} \rvert + k$ and $\lvert F_{m} \rvert \leq \lvert F^{-}_{m} \rvert + k$). By the definition of $F^{-}_{m}$, $F^{-}_{m} \cap V(A_{1}) = \emptyset$, and so
		\begin{equation}
			\begin{split}
				\lvert N^{-}_{L-A_{1}}(a) \cap F^{-}_{m} \rvert &\geq \lvert L - A_{1} \rvert + \lvert A_{1} \rvert - \lvert F^{-}_{m} \rvert - k + \lvert F^{-}_{m} \rvert - \lvert L - A_{1} \rvert \\
				&= \lvert A_{1} \rvert - k = (n - 110k^{3}) / (110k^{2}). \label{equ:intersection}
			\end{split}
		\end{equation}
		
		Let $I_{1} := sP_{1}t$ and $I_{i} := P_{i} \setminus \{t_{i}\}$ ($i \neq 1$). We use $G_{i}$ to denote the auxiliary bipartite graph whose vertex sets are $V(A_{2})$ and $V(I_{i}) \cap F^{-}_{m}$ ($i = 1, \dots , k$). For any $a \in V(A_{2})$ and $w \in V(I_{i}) \cap F^{-}_{m}$, if $wa \in A(D)$, then there is an edge between $a$ and $w$ in each $G_{i}$. As $F^{-}_{m} \cap V(tP_{1}t_{1}) = \emptyset$, $F^{-}_{m} \subseteq V(I_{1}) \cup \dots \cup V(I_{k})$ and the edges of $G_{1}\cup \dots \cup G_{k}$ are equivalent to the arcs which are from $F^{-}_{m}$ to $A_{2}$ in $D$. Since $\lvert N^{-}_{L-A_{1}}(a) \cap F^{-}_{m} \rvert \geq (n - 110k^{3}) / (110k^{2})$, there exists a $G_i$ satisfying $$e(G_{i}) \geq \frac{\lvert A_{2} \rvert (n - 110k^{3})}{110k^{3}} \geq \frac{n(n - 110k^{3})}{12100k^{5}} \geq 3n \geq 3\lvert G_{i} \rvert.$$ 
		This implies that $G_{i}$ is not planar, so there are vertices $a_{1}$, $a_{2} \in V(A_{2})$ and $w_{1}$, $w_{2} \in V(I_{i}) \cap F^{-}_{m}$ such that the edges $w_{1}a_{1}$, $w_{2}a_{2}$ cross in $G_{i}$.
		
		We first consider the case that $i = 1$, and then $w_{1}$, $w_{2} \in V(sP_{1}t) \cap F^{-}_{m}$ and $w^{+}_{1}$, $w^{+}_{2} \in F_{m}$. Together with the fact that $a^{-}_{1}$, $a^{-}_{2} \in R_{m}$, we can find disjoint paths $a^{-}_{j}Q_{j}w^{+}_{j}$ ($j = 1, 2$), where the vertices of $Q_{j}$ lie in $H$ and $\lvert Q_{1} \cup Q_{2} \rvert \geq 2$. Particularly, when $\lvert H \rvert = 2$, say $V(H) = \{u, v\}$, we have $Q_{1} = u$ and $Q_{2}= v$, or $Q_{1}= v$ and $Q_{2} = u$. Let $P^{*}_{1} := sP_{1}w_{1}a_{1}P_{1}a^{-}_{2}Q_{2}w^{+}_{2}P_{1}a^{-}_{1}Q_{1}w^{+}_{1}P_{1}w_{2}a_{2}P_{1}t_{1}$. Thus we obtain a larger $S$-$T$ path $L^{*}$ which contains at least $\lvert L \rvert + 2$ vertices, this produces a contradiction (see Figure~\ref{fig:bipar}).
		
		\begin{figure}[htb]
			\centering
			\begin{tikzpicture}
				
				\filldraw[black]    (0, 0)  circle (3pt)  node [anchor=east] {$s$};
				\filldraw[black]    (1.2, 0)  circle (3pt)  node [anchor=south] {$w_{1}$};
				\filldraw[black]    (2, 0)  circle (3pt)  node [anchor=north] {$w^{+}_{1}$};
				\filldraw[black]    (3.2, 0)  circle (3pt)  node [anchor=south] {$w_{2}$};
				\filldraw[black]    (4, 0)  circle (3pt)  node [anchor=north] {$w^{+}_{2}$};
				\filldraw[black]    (5.2, 0)  circle (3pt)  node [anchor=north] {$a^{-}_{1}$};
				\filldraw[black]    (6, 0)  circle (3pt)  node [anchor=south] {$a_{1}$};
				\filldraw[black]    (7.2, 0)  circle (3pt)  node [anchor=north] {$a^{-}_{2}$};
				\filldraw[black]    (8, 0)  circle (3pt)  node [anchor=south] {$a_{2}$};
				\filldraw[black]    (9.2, 0)  circle (3pt)  node [anchor=west] {$t_{1}$};
				
				\draw[line width=2pt] [-latex]      (1.2, 0) .. controls (3, -1.8) and (4.2, -1.8) .. (6, 0) ;
				\draw[line width=2pt] [-latex]     (3.2, 0) .. controls (5, -1.8) and (6.2, -1.8) .. (8, 0) ;
				\draw[line width=2pt] [-latex]    (5.2, 0) .. controls (4, 1.5) and (3.2, 1.5) .. (2, 0) node[midway, above] {$Q_{1}$};
				\draw[line width=2pt] [-latex]     (7.2, 0) .. controls (6, 1.5) and (5.2, 1.5) .. (4, 0) node[midway, above] {$Q_{2}$};
				
				\draw[line width=2pt] [-latex]     (0, 0) -- (1.2, 0) ;
				\draw[] [-latex]     (1.2, 0) -- (2, 0) ;
				\draw[line width=2pt] [-latex]      (2, 0) -- (3.2, 0) ;
				\draw[] [-latex]      (3.2, 0) -- (4, 0) ;
				\draw[line width=2pt] [-latex]      (4, 0) -- (5.2, 0) ;
				\draw[] [-latex]      (5.2, 0) -- (6, 0) ;
				\draw[line width=2pt] [-latex]      (6, 0) -- (7.2, 0) ;
				\draw[] [-latex]      (7.2, 0) -- (8, 0) ;
				\draw[line width=2pt] [-latex]      (8, 0) -- (9.2, 0) ;
				
			\end{tikzpicture}
			
			\caption{}
			\label{fig:bipar}
		\end{figure}
		
		We now consider the case that $i \neq 1$, without loss of generality, assume that $i = 2$. Consequently, $w_{1}$, $w_{2} \in V(P_{2} \setminus \{t_{2}\}) \cap F^{-}_{m}$ and $w^{+}_{1}$, $w^{+}_{2} \in F_{m}$. According to the argument above, $w_{1}a_{1}$ and $w_{2}a_{2}$ cross in $G_{2}$. Without loss of generality, assume that $a_{1}$ precedes $a_{2}$ in $P_{1}$ and $w_{2}$ precedes $w_{1}$ in $P_{2}$.
		By Theorem~\ref{The:Redei}, $H$ has a Hamiltonian path. Suppose that there is a Hamiltonian path from $u$ to $v$ in $H$. Since $a^{-}_{2} \in R_{m}$ and $w^{+}_{2} \in F_{m}$, we have $a^{-}_{2}u$, $vw^{+}_{2} \in A(D)$, which implies that there exists a path $a^{-}_{2}Qw^{+}_{2}$, where $Q$ contains all vertices in $H$. Let $P^{*}_{1} := sP_{2}w_{2}a_{2}P_{1}t_{1}$ and $P^{*}_{2} := sP_{1}a^{-}_{2}Qw^{+}_{2}P_{2}t_{2}$. Therefore, we obtain a larger $S$-$T$ path $L^{*}$ which is a one-to-many $k$-DDPC in $D$. This produces a contradiction (see Figure~\ref{fig:wpt}).
		
		\begin{figure}[htb]
			\centering
			\begin{tikzpicture}
				\filldraw[black]    (0, 0)  circle (3pt)  node [anchor=east] {$s$};
				\filldraw[black]    (4, 1.05)  circle (3pt)  node [anchor=south] {$a^{-}_{2}$};
				\filldraw[black]    (5, 1.1)  circle (3pt)  node [anchor=south] {$a_{2}$};
				\filldraw[black]    (8, 1.2)  circle (3pt)  node [anchor=west] {$t_{1}$};
				\filldraw[black]    (4, -1.05)  circle (3pt)  node [anchor=north] {$w_{2}$};
				\filldraw[black]    (5, -1.1)  circle (3pt)  node [anchor=north] {$w^{+}_{2}$};
				\filldraw[black]    (8, -1.2)  circle (3pt)  node [anchor=west] {$t_{2}$};
				
				\draw[line width=2pt] [-latex]      (0, 0) .. controls (0.6, 1.05) and (3, 1.05) .. (4, 1.05) ;
				\draw[line width=2pt] [-latex]     (4, 1.05) .. controls (4.3, 0.3) and (4.3, -0.2) .. (5, -1.1) node[near start, left] {$Q$};
				\draw[line width=2pt] [-latex]     (0, 0) .. controls (0.6, -1.05) and (3, -1.05) .. (4, -1.05) ;
				\draw[line width=2pt] [-latex]     (4, -1.05) .. controls (4.5, -0.2) and (4.5, -0.1) .. (5, 1.1) ;
				
				\draw[] [-latex]     (4, 1.05) -- (5, 1.1) ;
				\draw[line width=2pt] [-latex]      (5, 1.1) -- (8, 1.2) ;
				\draw[] [-latex]     (4, -1.05) -- (5, -1.1) ;
				\draw[line width=2pt] [-latex]      (5, -1.1) -- (8, -1.2) ;
				
			\end{tikzpicture}
			
			\caption{}
			\label{fig:wpt}
		\end{figure}
		
		\begin{case}
			$n / 2 - n / (50k) \leq \lvert H \rvert \leq \lceil (n-k) / 2\rceil - 1 = \lfloor(n-k+1) / 2\rfloor - 1$.
		\end{case}
		
		By Lemma~\ref{Lem:bound}, $\lvert R_{m} \rvert \geq (n-k+1) / 2 - \lvert H \rvert \geq 1$. Similarly, $\lvert F_{m} \rvert \geq 1$. By Lemma~\ref{Lem:FpR}, $\lvert R \cup F \rvert = \lvert L \rvert$, $\lvert R \cap F \rvert \leq k$ and $\lvert L \rvert - k \leq \lvert R_{m} \cup F_{m} \rvert = \lvert R_{m} \rvert + \lvert F_{m} \rvert \leq \lvert L \rvert$. Since $\lceil (n+k-1) / 2 \rceil + 1 \leq \lvert L \rvert \leq n / 2 + n / (50k)$, we deduce that $\lceil (n-k+1) / 2\rceil \leq \lvert R_{m} \rvert + \lvert F_{m} \rvert \leq n / 2 + n / (50k)$.
		
		Note that each $h \in V(H)$ satisfies
		\begin{equation*}
			\begin{split}
				d^{-}_{L}(h) \geq \delta^{-}(D) - (\lvert H \rvert - 1) &\geq \lceil(n+k-1) / 2\rceil - \lceil(n-k) / 2\rceil + 1 + 1 \\
				&\geq (n+k-1) / 2 - (n-k+1) / 2 + 2 = k + 1
			\end{split}
		\end{equation*}
		and so $\lvert R \rvert \geq k + 1$. Similarly, $d^{+}_{L}(h) \geq \delta^{+}(D) - (\lvert H \rvert - 1) \geq k + 1$,  so $\lvert F \rvert \geq k + 1$.
		For any vertex $x \in R_{m}$, Lemma~\ref{Lem:FpR} implies that $N^{-}_{D}(x) \subseteq V(L)$. Furthermore,
		$$\lvert L \rvert - \delta^{-}(D) \leq n / 2 + n / (50k) - (n+k-1) / 2 < n / (50k),$$
		so for every vertex $x \in R_{m}$ and a vertex set $Z_{1} \subseteq V(L)$ satisfying $\lvert Z_{1} \rvert \geq n / (50k)$, there exists a vertex $a_{1} \in Z_{1}$ such that $a_{1}x \in A(D)$. Similarly, for every vertex $y \in F_{m}$ and a vertex set $Z_{2} \subseteq V(L)$ satisfying $\lvert Z_{2} \rvert \geq n / (50k)$, there exists a vertex $a_{2} \in Z_{2}$ such that $ya_{2} \in A(D)$.
        
		\begin{case 3.1}
			$1 \leq \lvert R_{m} \rvert \leq k$.
		\end{case 3.1}
		By Lemma~\ref{Lem:FpR}, $\lvert (R \cap F) \cap V(P_{i}) \rvert \leq 1$. Together with the fact that $\lvert R \rvert \geq k + 1$, we deduce that there exists a path $P_{i}$ such that $x_{1} \in R$ and $x^{+}_{1} \in R_{m}$. Observe that $\lvert F_{m} \rvert \geq \lceil (n-k+1) / 2\rceil - k \geq n / 3$ since $\lvert R_{m} \rvert \leq k$, and so one of the $k$ paths of $L$, say $P_{j}$, contains at least $n / (3k)$ vertices from $F_{m}$. Thus there is a subpath $y_{1}P_{j}y_{2}$ on path $P_{j}$ such that $V(y_{1}P_{j}y_{2}) \subseteq F_{m}$ and $\lvert y_{1}P_{j}y_{2} \rvert = n / (20k)$. 
		
		We first consider the case that $j \neq i$, without loss of generality, assume that $i=1$ and $j=2$. There is a vertex $a_{1} \in V(y_{1}P_{2}y^{-}_{2})$ such that $a_{1}x^{+}_{1} \in A(D)$ since $x^{+}_{1} \in R_{m}$ and $\lvert y_{1}P_{2}y^{-}_{2} \rvert \geq n / (50k)$. As $x_{1} \in R$, there exists a vertex $h_{1} \in V(H)$ such that $x_{1}h_{1} \in A(D)$. According to the argument in Case 2.1, we deduce that $H$ has a Hamiltonian cycle. As $a^{+}_{1} \in F_{m}$, $h^{-}_{1}a^{+}_{1} \in A(D)$. Note that there is a Hamiltonian path from $h_{1}$ to $h^{-}_{1}$ in $H$. This implies that there exists a path $x_{1}Q_{1}a^{+}_{1}$, where $Q_{1}$ contains all vertices in $H$. Let $P^{*}_{1} := sP_{2}a_{1}x^{+}_{1}P_{1}t_{1}$ and $P^{*}_{2} := sP_{1}x_{1}Q_{1}a^{+}_{1}P_{2}t_{2}$. Now we obtain a larger $S$-$T$ path $L^{*}$ which is a one-to-many $k$-DDPC in $D$. This produces a contradiction (see Figure~\ref{fig:jneqi}).
		
		\begin{figure}[htb]
			\centering
			\begin{tikzpicture}
				\filldraw[black]    (0, 0)  circle (3pt)  node [anchor=east] {$s$};
				\filldraw[black]    (6, 1.15)  circle (3pt)  node [anchor=south] {$x_{1}$};
				\filldraw[black]    (7, 1.2)  circle (3pt)  node [anchor=south] {$x^{+}_{1}$};
				\filldraw[black]    (8, 1.2)  circle (3pt)  node [anchor=west] {$t_{1}$};
				\filldraw[black]    (4, -1.05)  circle (3pt)  node [anchor=north] {$a_{1}$};
				\filldraw[black]    (5, -1.1)  circle (3pt)  node [anchor=north] {$a^{+}_{1}$};
				\filldraw[black]    (8, -1.2)  circle (3pt)  node [anchor=west] {$t_{2}$};
				
				\draw[line width=2pt] [-latex]      (0,0) .. controls (0.5, 1.15) and (4, 1.15) .. (6, 1.15) ;
				\draw[line width=2pt] [-latex]     (0, 0) .. controls (0.5, -1.05) and (3, -1.05) .. (4, -1.05) ;
				\draw[line width=2pt] [-latex]     (4, -1.05) -- (7, 1.2) ;
				\draw[line width=2pt] [-latex]     (6, 1.15) -- (5, -1.1) node[near start, left] {$Q_{1}$};
				
				\draw[] [-latex]     (6, 1.15) -- (7, 1.2) ;
				\draw[line width=2pt] [-latex]      (7, 1.2) -- (8, 1.2) ;
				\draw[] [-latex]     (4, -1.05) -- (5, -1.1) ;
				\draw[line width=2pt] [-latex]     (5, -1.1) -- (8, -1.2) ;
				
			\end{tikzpicture}
			
			\caption{}
			\label{fig:jneqi}
		\end{figure}
		
		We next consider the case that $j = i$, without loss of generality, assume that $j=i=1$. Let $I_{1} := y_{1}P_{1}y_{3}$, $I_{2} := y^{+}_{3}P_{1}y_{4}$, $I_{3} := y^{+}_{4}P_{1}y_{5}$, and $I_{4} := y^{+}_{5}P_{1}y_{2}$, where $\lvert I_{1} \rvert = n / (240k)$, $\lvert I_{2} \rvert = n / (240k)$, $\lvert I_{3} \rvert  = n / (48k)$, and $\lvert I_{4} \rvert = n / (48k)$. According to the argument at the beginning of Case 3, there is a vertex $a_{3} \in V(I_{3})$ such that $a_{3}x^{+}_{1} \in A(D)$ since $x^{+}_{1} \in R_{m}$ and $\lvert I_{3} \rvert > n / (50k)$. Similarly, there are two vertices $a_{1} \in I_{1}$ and $a_{4} \in I_{4}$ such that $a_{1}a_{4} \in A(D)$ (since $a_{1} \in F_{m}$ and $\lvert I_{4} \rvert > n / (50k)$). As $x_{1} \in R$, there exists a vertex $h_{1} \in V(H)$ such that $x_{1}h_{1} \in A(D)$. For a vertex $a_{2} \in I_{2}$, we have $a_{2} \in F_{m}$, so $h^{-}_{1}a_{2} \in A(D)$. According to the argument in Case 2.1, we deduce that $H$ contains a Hamiltonian cycle, and so there is a Hamiltonian path from $h_{1}$ to $h^{-}_{1}$ in $H$. This implies that there exists a path $x_{1}Q_{2}a_{2}$, where $Q_{2}$ contains all vertices in $H$. Let $P^{*}_{1} := sP_{1}a_{1}a_{4}P_{1}x_{1}Q_{2}a_{2}P_{1}a_{3}x^{+}_{1}P_{1}t_{1}$. Thus we obtain a larger $S$-$T$ path $L^{*}$ which contains at least $\lvert H \rvert - \lvert I_{1} \rvert - \lvert I_{2} \rvert - \lvert I_{3} \rvert - \lvert I_{4} \rvert \geq n / 2 - n / (50k) - n / (20k) > 0$ vertices more than $L$. This produces a contradiction (see Figure~\ref{fig:jequali}).
		
		\begin{figure}[htb]
			\centering
			\begin{tikzpicture}
				
				\filldraw[black]    (0, 0)  circle (3pt)  node [anchor=north] {$s$};
				\filldraw[black]    (1.5, 0)  circle (3pt)  node [anchor=south] {$a_{1}$};
				\filldraw[black]    (2.5, 0)  circle (3pt)  node [anchor=north] {$a_{2}$};
				\filldraw[black]    (3.7, 0)  circle (3pt)  node [anchor=south] {$a_{3}$};
				\filldraw[black]    (5, 0)  circle (3pt)  node [anchor=south] {$a_{4}$};
				\filldraw[black]    (6.2, 0)  circle (3pt)  node [anchor=north] {$x_{1}$};
				\filldraw[black]    (7.2, 0)  circle (3pt)  node [anchor=south] {$x^{+}_{1}$};
				\filldraw[black]    (8, 0)  circle (3pt)  node [anchor=west] {$t_{1}$};
				
				\draw[line width=2pt] [-latex]      (1.5, 0) .. controls (2.5, -1.3) and (4, -1.3) .. (5, 0) ;
				\draw[line width=2pt] [-latex]     (3.7, 0) .. controls (4.7, -1.3) and (6.2, -1.3) .. (7.2, 0) ;
				\draw[line width=2pt] [-latex]     (6.2, 0) .. controls (5, 1.3) and (3.7, 1.3) .. (2.5, 0) node[midway, above] {$Q_{2}$};
				
				\draw[line width=2pt] [-latex]     (0, 0) -- (1.5, 0) ;
				\draw[] [-latex]     (1.5, 0) -- (2.5, 0) ;
				\draw[line width=2pt] [-latex]      (2.5, 0) -- (3.7, 0) ;
				\draw[] [-latex]      (3.7, 0) -- (5, 0) ;
				\draw[line width=2pt] [-latex]      (5, 0) -- (6.2, 0) ;
				\draw[] [-latex]      (6.2, 0) -- (7.2, 0) ;
				\draw[line width=2pt] [-latex]      (7.2, 0) -- (8, 0) ;
				
			\end{tikzpicture}
			
			\caption{}
			\label{fig:jequali}
		\end{figure}
		
		\begin{case 3.2}
			$\lvert F_{m} \rvert = 1$.
		\end{case 3.2}
		In this case, $n-k$ is odd, $\lvert H \rvert = (n-k+1) / 2 - 1$ and $\lvert L \rvert = (n+k-1) / 2 + 1$ by Lemma~\ref{Lem:bound}. Recall that $\lvert F \rvert \geq k + 1$ at the beginning of Case 3.
		By Lemma~\ref{Lem:FpR}, $\lvert R \cap F\rvert \leq k$ and $\lvert (R \cap F) \cap V(P_{j}) \rvert \leq 1$.  Therefore, $s \in F_{m}$ and there exists exactly a vertex $s^{+}_{j}$ in each $P_{j}$ such that $s^{+}_{j} \in R \cap F$ by Lemma~\ref{Lem:FpR}. As $\lvert R_{m} \rvert + \lvert F_{m} \rvert \geq \lceil (n-k+1) / 2 \rceil$, $\lvert R_{m} \rvert \geq \lceil (n-k+1) / 2 \rceil - 1\geq n / 3$. Consequently, one of the $k$ paths of $L$, say $P_{i}$, contains at least $n / (3k)$ vertices from $R_{m}$. There is a subpath $x_{1}P_{i}x_{2}$ on $P_{i}$ such that $V(x_{1}P_{i}x_{2}) \subseteq R_{m}$ and $\lvert x_{1}P_{i}x_{2} \rvert = n / (20k)$.
		
		Without loss of generality, assume that $i = 1$. There are two vertices $s \in F_{m}$ and $s^{+}_{2} \in R \cap F$ in $P_{2}$. According to the argument at the beginning of Case 3, there is a vertex $a_{1} \in V(x^{+}_{1}P_{1}x_{2})$ such that $sa_{1} \in A(D)$ since $s \in F_{m}$ and $\lvert x^{+}_{1}P_{1}x_{2} \rvert > n / (50k)$. As $s^{+}_{2} \in F$, there exists a vertex $h_{1} \in V(H)$ such that $h_{1}s^{+}_{2} \in A(D)$. By the assumption, $a^{-}_{1} \in R_{m}$, so there is an arc from $a^{-}_{1}$ to $h^{+}_{1}$. According to the argument in Case 2.1, we deduce that  $H$ contains a Hamiltonian cycle. Note that there is a 
		Hamiltonian path from $h^{+}_{1}$ to $h_{1}$ in $H$. This implies that there exists a path $a^{-}_{1}Qs^{+}_{2}$, where $Q$ contains all vertices in $H$. Let $P^{*}_{1} := sa_{1}P_{1}t_{1}$ and $P^{*}_{2} := sP_{1}a^{-}_{1}Qs^{+}_{2}P_{2}t_{2}$. Thus we obtain a larger $S$-$T$ path $L^{*}$ which is a one-to-many $k$-DDPC in $D$. This produces a contradiction (see Figure~\ref{fig:fmh}).

		\begin{figure}[htb]
			\centering
			\begin{tikzpicture}
				\filldraw[black]    (0,0)  circle (3pt)  node [anchor=east] {$s$};
				\filldraw[black]    (5, 1.1)  circle (3pt)  node [anchor=north] {$a^{-}_{1}$};
				\filldraw[black]    (6, 1.15)  circle (3pt)  node [anchor=north] {$a_{1}$};
				\filldraw[black]    (8, 1.2)  circle (3pt)  node [anchor=west] {$t_{1}$};
				\filldraw[black]    (2, -1)  circle (3pt)  node [anchor=north] {$s^{+}_{2}$};
				\filldraw[black]    (8, -1.2)  circle (3pt)  node [anchor=west] {$t_{2}$};
				
				\draw[line width=2pt] [-latex]     (0, 0) .. controls (1, 1.05) and (4, 1.05) .. (5, 1.1) ;
				\draw[line width=2pt] [-latex]      (6, 1.15) .. controls (6.6, 1.2) and (7.5, 1.2) .. (8, 1.2) ;
				\draw[] [-latex]     (0, 0) .. controls (0.5, -1) and (1.5, -1) .. (2, -1) ;
				\draw[line width=2pt] [-latex]     (2, -1) .. controls (4, -1.2) and (6, -1.2) .. (8, -1.2) ;
				\draw[line width=2pt] [-latex]     (0, 0) .. controls (2, 2) and (4, 2) .. (6, 1.15) ;
				\draw[line width=2pt] [-latex]     (5, 1.1) .. controls (4.5, 0.5) and (3, -0.5) .. (2, -1) node[midway, right] {$Q$};
				
				\draw[] [-latex]     (5, 1.1) -- (6, 1.15) ;
				
			\end{tikzpicture}
			\caption{}
			\label{fig:fmh}
		\end{figure}
		
		\begin{case 3.3}
			$\lvert F_{m} \rvert \geq 2$ and $\lvert R_{m} \rvert \geq k + 1$.
		\end{case 3.3}
		Since $\lvert R_{m} \rvert \geq k + 1$, there is a path $P_{i}$ such that $x_{1}, x_{2} \in R_{m}$, and $x_{1}$ precedes $x_{2}$. Similarly, as $\lvert F_{m} \rvert \geq 2$, there is a path $P_{j}$ such that $y_{1}, y_{2} \in F_{m}$, and $y_{1}$ precedes $y_{2}$.
		
		\begin{case 3.3.1}
			There exists a subpath $x_{1}P_{i}x_{2}$ on $P_{i}$ such that $\lvert x_{1}P_{i}x_{2} \rvert \geq n /(20k)$.
		\end{case 3.3.1}
		We first consider the case that $j \neq i$, without loss of generality, assume that $i = 1$, $j = 2$, and $y_{1} = y^{-}_{2}$. Lemma~\ref{Lem:FpR} implies that $V(x_{1}P_{i}x_{2}) \subseteq R_{m}$.
		Let $A_{1} := x_{1}P_{1}x_{3}$ and $A_{2} := x^{+}_{3}P_{1}x_{4}$, where $\lvert A_{1} \rvert = n / (48k)$ and $\lvert A_{2} \rvert = n / (48k)$. According to the argument at the beginning of Case 3, there is a vertex $a_{1} \in A_{1}$ such that $a_{1}x_{2} \in A(D)$ since $x_{2} \in R_{m}$ and $\lvert A_{1} \rvert > n / (50k)$. Similarly, as $y^{-}_{2} \in F_{m}$ and $\lvert A_{2} \rvert > n / (50k)$, there exists a vertex $a_{2} \in A_{2}$ such that $y^{-}_{2}a_{2} \in A(D)$. By Theorem~\ref{The:Redei}, $H$ has a Hamiltonian path. Suppose that there is a Hamiltonian path from $u$ to $v$ in $H$. Since $x^{-}_{2} \in R_{m}$ and $y_{2} \in F_{m}$, $x^{-}_{2}u$, $vy_{2} \in A(D)$. This implies that there exists a path $x^{-}_{2}Q_{1}y_{2}$, where $Q_{1}$ contains all vertices in $H$. Let $P^{*}_{1} := sP_{1}a_{1}x_{2}P_{1}t_{1}$ and $P^{*}_{2} := sP_{2}y^{-}_{2}a_{2}P_{1}x^{-}_{2}Q_{1}y_{2}P_{2}t_{2}$. Thus we obtain a larger $S$-$T$ path $L^{*}$ which contains at least $\lvert H \rvert - \lvert V(a_{1}P_{1}a_{2}) \rvert \geq n / 2 - n / (50k) - n / (24k) > 0$ vertices more than $L$. This produces a contradiction (see Figure~\ref{fig:xpx}).
		
		\begin{figure}[htb]
			\centering
			\begin{tikzpicture}
				\filldraw[black]    (0, 0)  circle (3pt)  node [anchor=east] {$s$};
				\filldraw[black]    (3, 0.8)  circle (3pt)  node [anchor=north] {$a_{1}$};
				\filldraw[black]    (4.5, 0.9)  circle (3pt)  node [anchor=south] {$a_{2}$};
				\filldraw[black]    (6, 1)  circle (3pt)  node [anchor=south] {$x^{-}_{2}$};
				\filldraw[black]    (7, 1)  circle (3pt)  node [anchor=north] {$x_{2}$};
				\filldraw[black]    (8, 1)  circle (3pt)  node [anchor=west] {$t_{1}$};
				\filldraw[black]    (2, -0.85)  circle (3pt)  node [anchor=north] {$y^{-}_{2}$};
				\filldraw[black]    (3,-0.9)  circle (3pt)  node [anchor=north] {$y_{2}$};
				\filldraw[black]    (8, -1)  circle (3pt)  node [anchor=west] {$t_{2}$};
				
				\draw[line width=2pt] [-latex]      (0, 0) .. controls (0.5, 0.8) and (2, 0.8) .. (3, 0.8) ;
				\draw[line width=2pt] [-latex]    (0, 0) .. controls (0.5, -0.85) and (1.5, -0.85) .. (2, -0.85) ;
				\draw[line width=2pt] [-latex]    (3, 0.8) .. controls (4, 2) and (6, 2) .. (7, 1) ;
				\draw[line width=2pt] [-latex]    (2, -0.85) .. controls (2.3, -0.5) and (2.7, 0.3) .. (4.5, 0.9) ;
				\draw[line width=2pt] [-latex]    (6, 1) .. controls (5.7, 0.5) and (5, -0.3) .. (3, -0.9) node[midway, right] {$Q_{1}$};
				
				\draw[] [-latex]     (3, 0.8) -- (4.5, 0.9) ;
				\draw[line width=2pt] [-latex]    (4.5, 0.9) -- (6, 1) ;
				\draw[] [-latex]    (6, 1) -- (7, 1) ;
				\draw[line width=2pt] [-latex]    (7, 1) -- (8, 1) ;
				\draw[] [-latex]    (2, -0.85) -- (3, -0.9) ;
				\draw[line width=2pt] [-latex]    (3, -0.9) -- (8, -1) ;
				
			\end{tikzpicture}
			\caption{}
			\label{fig:xpx}
		\end{figure}
		
		We next consider the case that $j = i$, without loss of generality, assume that $j = i = 1$ and $y_{1} = y^{-}_{2}$. The argument for the case that $j = i$ is similar to that of the case $j \neq i$, so there exists a path $x^{-}_{2}Q_{2}y_{2}$, where $Q_{2}$ contains all vertices in $H$. Let $P^{*}_{1} := sP_{1}y^{-}_{2}a_{2}P_{1}x^{-}_{2}Q_{2}y_{2}P_{1}a_{1}x_{2}P_{1}t_{1}$. Thus we obtain a larger $S$-$T$ path $L^{*}$ which contains at least $\lvert H \rvert - \lvert V(a_{1}P_{1}a_{2}) \rvert \geq n / 2 - n / (50k) - n / (24k) > 0$ vertices more than $L$. This produces a contradiction (see Figure~\ref{fig:jione}).  
		
		\begin{figure}[htb]
			\centering
			\begin{tikzpicture}
				
				\filldraw[black]    (0, 0)  circle (3pt)  node [anchor=east] {$s$};
				\filldraw[black]    (1, 0)  circle (3pt)  node [anchor=south] {$y^{-}_{2}$};
				\filldraw[black]    (2, 0)  circle (3pt)  node [anchor=north] {$y_{2}$};
				\filldraw[black]    (3.5, 0)  circle (3pt)  node [anchor=south] {$a_{1}$};
				\filldraw[black]    (5, 0)  circle (3pt)  node [anchor=south] {$a_{2}$};
				\filldraw[black]    (6.5, 0)  circle (3pt)  node [anchor=north] {$x^{-}_{2}$};
				\filldraw[black]    (7.5, 0)  circle (3pt)  node [anchor=south] {$x_{2}$};
				\filldraw[black]    (8.5, 0)  circle (3pt)  node [anchor=west] {$t_{1}$};
				
				\draw[line width=2pt] [-latex]      (1, 0) .. controls (2.5, -1.5) and (3.5, -1.5) .. (5, 0) ;
				\draw[line width=2pt] [-latex]     (3.5, 0) .. controls (5, -1.5) and (6, -1.5) .. (7.5, 0) ;
				\draw[line width=2pt] [-latex]     (6.5, 0) .. controls (5, 1.5) and (3.5, 1.5) .. (2, 0) node[midway, above] {$Q_{2}$};
				
				\draw[line width=2pt] [-latex]     (0, 0) -- (1, 0);
				\draw[] [-latex]     (1, 0) -- (2, 0);
				\draw[line width=2pt] [-latex]      (2, 0) -- (3.5, 0);
				\draw[] [-latex]      (3.5, 0) -- (5, 0);
				\draw[line width=2pt] [-latex]      (5, 0) -- (6.5, 0);
				\draw[] [-latex]      (6.5, 0) -- (7.5, 0);
				\draw[line width=2pt] [-latex]      (7.5, 0) -- (8.5, 0);
				
			\end{tikzpicture}
			\caption{}
			\label{fig:jione}
		\end{figure}
		
		\begin{case 3.3.2}
		Case 3.3.1 does not hold.
		
		\end{case 3.3.2}
		There is not a subpath $x_{1}P_{i}x_{2}$ in each path $P_{i}$ such that $\lvert x_{1}P_{i}x_{2} \rvert \geq n / (20k)$, so $\lvert R_{m} \rvert < n / (20)$. Recall that $\lvert R_{m} \rvert + \lvert F_{m} \rvert \geq \lceil (n-k+1) / 2\rceil$ at the beginning of Case 3, so $\lvert F_{m} \rvert \geq \lceil (n-k+1) / 2\rceil - n / (20) \geq n / 3$. One of the $k$ paths of $L$, say $P_{j}$, contains at least $n / (3k)$ vertices from $F_{m}$. Without loss of generality, assume that $j = 2$. We can find a subpath $y_{1}P_{2}y_{2}$ such that $ V(y_{1}P_{2}y_{2}) \subseteq F_{m}$ and $\lvert V(y_{1}P_{2}y_{2}) \rvert \geq n / (20k)$. There is a path $P_{i}$ such that $x_{1}, x_{2} \in R_{m}$ and $x_{1}$ precedes $x_{2}$ since $\lvert R_{m} \rvert \geq k + 1$. 
		
		We first consider the case that $j \neq i$, without loss of generality, assume that $i = 1$, $j = 2$, and $x_{2} = x^{+}_{1}$.
		Let $A_{2} := y_{3}P_{2}y_{2}$ and $A_{1} := y_{4}P_{2}y^{-}_{3}$, where $\lvert A_{2} \rvert = n / (48k)$ and $\lvert A_{1} \rvert = n / (48k)$. According to the argument at the beginning of Case 3, there is a vertex $a_{2} \in A_{2}$ such that $y_{1}a_{2} \in A(D)$ since $y_{1} \in F_{m}$ and $\lvert A_{2} \rvert > n / (50k)$. Similarly, as $x^{+}_{1} \in R_{m}$ and $\lvert A_{1} \rvert > n / (50k)$, there exists a vertex $a_{1} \in A_{1}$ such that $a_{1}x^{+}_{1} \in A(D)$. By Theorem~\ref{The:Redei}, $H$ has a Hamiltonian path. Assume that there is a Hamiltonian path from $u$ to $v$ in $H$. Since $x_{1} \in R_{m}$ and $y^{+}_{1} \in F_{m}$, $x_{1}u$, $vy^{+}_{1} \in A(D)$. This implies that there exists a path $x_{1}Q_{1}y^{+}_{1}$, where $Q_{1}$ contains all vertices in $H$. Let $P^{*}_{1} := sP_{1}x_{1}Q_{1}y^{+}_{1}P_{2}a_{1}x^{+}_{1}P_{1}t_{1}$ and $P^{*}_{2} := sP_{2}y_{1}a_{2}P_{2}t_{2}$. Thus we obtain a larger $S$-$T$ path $L^{*}$ which contains at least $\lvert H \rvert - \lvert V(a_{1}P_{2}a_{2}) \rvert \geq n / 2 - n / (50k) - n / (24k) > 0$ vertices more than $L$. This produces a contradiction (see Figure~\ref{fig:holdtwo}).
		
		\begin{figure}[htb]
			\centering
			\begin{tikzpicture}
				\filldraw[black]    (0, 0)  circle (3pt)  node [anchor=east] {$s$};
				\filldraw[black]    (5, 0.95)  circle (3pt)  node [anchor=south] {$x_{1}$};
				\filldraw[black]    (6, 1)  circle (3pt)  node [anchor=south] {$x^{+}_{1}$};
				\filldraw[black]    (8, 1.05)  circle (3pt)  node [anchor=west] {$t_{1}$};
				\filldraw[black]    (2, -0.85)  circle (3pt)  node [anchor=south] {$y_{1}$};
				\filldraw[black]    (3, -0.9)  circle (3pt)  node [anchor=north] {$y^{+}_{1}$};
				\filldraw[black]    (4.5, -0.95)  circle (3pt)  node [anchor=north] {$a_{1}$};
				\filldraw[black]    (6, -1)  circle (3pt)  node [anchor=south] {$a_{2}$};
				\filldraw[black]    (8, -1.05)  circle (3pt)  node [anchor=west] {$t_{2}$};
				
				\draw[line width=2pt] [-latex]    (0, 0) .. controls (0.7, 0.95) and (3, 0.95) .. (5, 0.95) ;
				\draw[line width=2pt] [-latex]      (0,0) .. controls (0.7, -0.85) and (1.2, -0.85) .. (2, -0.85) ;
				\draw[line width=2pt] [-latex]     (2, -0.85) .. controls (2.7, -2) and (5.3, -2) .. (6, -1) ;
				\draw[line width=2pt] [-latex]    (4.5, -0.95) .. controls (5.3, -0.2) and (5.7, 0.3) .. (6, 1) ;
				\draw[line width=2pt] [-latex]    (5, 0.95) .. controls (4, 0.3) and (3.5, -0.3) .. (3, -0.9) node[midway, left] {$Q_{1}$};
				
				\draw[] [-latex]    (5, 0.95) -- (6, 1) ;
				\draw[line width=2pt] [-latex]    (6, 1) -- (8, 1.05) ;
				\draw[] [-latex]    (2, -0.85) -- (3, -0.9) ;
				\draw[line width=2pt] [-latex]    (3, -0.9) -- (4.5, -0.95) ;
				\draw[] [-latex]    (4.5, -0.95) -- (6, -1) ;
				\draw[line width=2pt] [-latex]     (6, -1) -- (8, -1.05) ;
				
			\end{tikzpicture}
			\caption{}
			\label{fig:holdtwo} 
		\end{figure}
		
		We next consider the case that $i = j = 2$, without loss of generality, assume that $x_{2} = x^{+}_{1}$.
		Arguing similarly as that of the case $i \neq j$, we get that there exists a path $x_{1}Q_{2}y^{+}_{1}$, where $Q_{2}$ contains all vertices in $H$. Let $P^{*}_{2} := sP_{2}y_{1}a_{2}P_{2}x_{1}Q_{2}y^{+}_{1}P_{2}a_{1}x^{+}_{1}P_{2}t_{2}$. Thus we obtain a larger $S$-$T$ path $L^{*}$ which contains at least $\lvert H \rvert - \lvert V(a_{1}P_{2}a_{2}) \rvert \geq n / 2 - n / (50k) - n / (24k) > 0$ vertices more than $L$. This produces a contradiction (see Figure~\ref{fig:holdone}). 
		
		\begin{figure}[htb]
			\centering
			\begin{tikzpicture}
				
				\filldraw[black]    (0, 0)  circle (3pt)  node [anchor=east] {$s$};
				\filldraw[black]    (1.5, 0)  circle (3pt)  node [anchor=south] {$y_{1}$};
				\filldraw[black]    (2.5, 0)  circle (3pt)  node [anchor=north] {$y^{+}_{1}$};
				\filldraw[black]    (4, 0)  circle (3pt)  node [anchor=south] {$a_{1}$};
				\filldraw[black]    (5.5, 0)  circle (3pt)  node [anchor=south] {$a_{2}$};
				\filldraw[black]    (7, 0)  circle (3pt)  node [anchor=north] {$x_{1}$};
				\filldraw[black]    (8, 0)  circle (3pt)  node [anchor=south] {$x^{+}_{1}$};
				\filldraw[black]    (9, 0)  circle (3pt)  node [anchor=west] {$t_{2}$};
				
				\draw[line width=2pt] [-latex]      (1.5, 0) .. controls (3, -1.5) and (4, -1.5) .. (5.5, 0) ;
				\draw[line width=2pt] [-latex]     (4, 0) .. controls (5.5, -1.5) and (6.5, -1.5) .. (8, 0) ;
				\draw[line width=2pt] [-latex]     (7, 0) .. controls (5.5, 1.5) and (4, 1.5) .. (2.5, 0) node[midway, above] {$Q_{2}$};
				
				\draw[line width=2pt] [-latex]     (0, 0) -- (1.5, 0);
				\draw[] [-latex]     (1.5, 0) -- (2.5, 0);
				\draw[line width=2pt] [-latex]      (2.5, 0) -- (4, 0);
				\draw[] [-latex]      (4, 0) -- (5.5, 0);
				\draw[line width=2pt] [-latex]      (5.5, 0) -- (7, 0);
				\draw[] [-latex]      (7, 0) -- (8, 0);
				\draw[line width=2pt] [-latex]      (8, 0) -- (9, 0);
				
			\end{tikzpicture}
			\caption{}
			\label{fig:holdone} 
			\end{figure}
			\end{proof}
		
\vskip 1cm

\noindent {\bf Acknowledgement.} Yuefang Sun was supported by Yongjiang Talent Introduction Programme of Ningbo under Grant No. 2021B-011-G and Zhejiang Provincial Natural Science
Foundation of China under Grant No. LY20A010013. Xiaoyan Zhang was supported by NSFC under Grant No. 11871280.

\end{document}